\documentclass[12pt,reqno]{amsart}

\usepackage{a4}

\usepackage{amssymb}

\usepackage{enumitem}

\usepackage[dvipsnames]{xcolor}
\usepackage[mathscr]{euscript}

\usepackage[colorlinks=true,urlcolor=blue,linkcolor=blue,citecolor=blue]{hyperref}

\theoremstyle{plain}
\newtheorem{thm}{Theorem}[section]
\newtheorem*{mainthm}{Main Theorem}
\newtheorem{lemma}[thm]{Lemma}
\newtheorem{prop}[thm]{Proposition}
\newtheorem{cor}[thm]{Corollary}
\newtheorem{claim}{Claim}

\theoremstyle{definition}

\newtheorem{example}[thm]{Example}

\theoremstyle{remark}
\newtheorem{rem}[thm]{Remark}

\newenvironment{thmenumerate}{
\begin{enumerate}[label=\textup{(\roman*)}, widest=(ii), leftmargin=9mm,itemsep=1mm,topsep=0mm]}{
\end{enumerate}}

\newcommand{\A}{\mathcal{A}}
\newcommand{\B}{\mathcal{B}}
\newcommand{\C}{\mathcal{C}}
\newcommand{\T}{\mathcal{T}}
\renewcommand{\S}{\mathcal{S}}

\newcommand{\N}{\mathbb{N}}
\newcommand{\Z}{\mathbb{Z}}

\newcommand{\x}{\mathbf{x}}
\newcommand{\y}{\mathbf{y}}
\renewcommand{\a}{\mathbf{a}}
\renewcommand{\t}{\mathbf{t}}
\newcommand{\0}{\mathbf{0}}
\newcommand{\1}{\mathbf{1}}

\DeclareMathOperator{\cont}{cont}
\DeclareMathOperator{\form}{form}
\DeclareMathOperator{\Mon}{Mon}
\DeclareMathOperator{\lcm}{lcm}
\DeclareMathOperator{\Seq}{Seq}

\newcommand{\fmin}{f_{\min}}

\title[Number of subdirect powers of unary algebras]{On the number of countable subdirect powers of unary algebras}

\author{Nik Ru\v{s}kuc}
\author{Bill de Witt}
\address{School of Mathematics and Statistics, University of St Andrews, St Andrews, Scotland, UK}
\email{$\{$bldw,nr1$\}$@st-andrews.ac.uk}


\date{\today}

\begin{document}

\maketitle

\begin{abstract}
A finite unary algebra $(A,F)$ has only countably many countable subdirect powers if and only if every operation $f\in F$ is either a permutation or a constant mapping.
 \end{abstract}

\section{Introduction and the statement of the main result}
\label{sec:intro}

Subdirect products have played an important role throughout the development of algebra, underpinned by Birkhoff's subdirect representation theorem \cite[Chapter 4.5]{alv87}. For groups they have been often deployed in the combinatorial aspects of the theory, e.g. theory of generators and defining relations;
see \cite{BHM:FPSP,BM:SFP} and references therein. In a recent development, Mayr and Ru\v{s}kuc \cite{MR19}
have placed these developments in a  more general algebraic context,
and there are offshoots of this work for lattices \cite{DMR20} and semigroups \cite{Cl20,clayton2020number}.

The starting point for our investigation was the result of Hickin and Plotkin \cite{hickin1981boolean} that a finitely generated non-abelian group has 
continuum many countable subdirect powers of countable size. Of course, a finitely generated abelian group has only countably many non-isomorphic countable subdirect powers, so the above result establishes a full classification. McKenzie \cite{mckenzie1982subdirect} strengthened this by proving that any non-abelian group $G$ has $2^\kappa$ non-isomorphic subdirect powers of cardinality $\kappa$, for every infinite cardinal $\kappa\geq |G|$.

It is then natural to consider the number of subdirect powers for some algebraic structures other than groups.
For instance, slightly surprisingly, in \cite{clayton2020number} it is proved that the free monogenic semigroup $\N=\{1,2,3,\dots\}$ already has uncountably many subdirect powers contained in $\N^2$.

In this paper we are going to consider this question for finite unary algebras.
By a \emph{unary algebra} we mean a structure $\A=(A,F)$, where $A$ is a set (known as the carrier set of $\A$) and $F$ is a set of unary operations on $A$.
A \emph{subdirect power} of $\A$ is a subalgebra $P$ of a cartesian product $\A^X$ which projects onto $A$ in every coordinate. 
We say that $P$ is a \emph{countable subdirect power} of $\A$ if $|P|\leq\aleph_0$.
We will prove:

\begin{mainthm}
Let $\A=(A,F)$ be a finite unary algebra. The number of non-isomorphic countable subdirect powers of $\A$ is countable if and only if every operation in $F$ is either a bijection or a constant mapping.
\end{mainthm}

We will refer to the \emph{type} of an algebra $\A$ as being countable or uncountable depending on whether $\A$ has  countably or uncountably many non-isomorphic countable subdirect powers.
Given a countable subdirect power $P\leq \A^X$, there is no a priori restriction on $|X|$, but
in fact $P$ is isomorphic to a subdirect power in $\A^\N$, and we will henceforth focus  on such powers.
Throughout the term `countable' is used to mean `finite or of size $\aleph_0$'.

We note that it follows from the Main Theorem that the unary algebras of countable type are precisely the \emph{minimal unary algebras} in the sense of
 \cite[Definition 2.14]{HM88}, which were first introduced into the literature by the name of \emph{permutational} by P\'{a}lfy and Pudl\'{a}k  \cite{PP80}.

\section{Preliminaries: unary algebras and graphs}
\label{sec:prelim}

Let $\A=(A,F)$ be a unary algebra.
The \emph{graph} $\Gamma(\A)$ of $\A$ is the directed graph with the set of vertices $A$, and a directed  edge $a\rightarrow f(a)$ for every $a\in A$, $f\in F$.
The underlying undirected graph of $\Gamma(\A)$ will be denoted by $\overline{\Gamma}(\A)$.
By a \emph{connected component} of $\A$ we will mean a connected component of $\overline{\Gamma}(\A)$.
Any such connected component is the carrier of a subalgebra of $\A$.
An isomorphism $\A\rightarrow\B$ between two unary algebras maps the connected components of $\A$ isomorphically onto the connected components of $\B$.
More precisely,
suppose that the connected components of $\A$ and $\B$ are
$\{ C^\A_i\::\: i\in I\}$, $\{C^\B_j\::\: j\in J\}$ respectively.
Then for any isomorphism $\phi:\A\rightarrow\B$ 
there exists a bijection $\pi:I\rightarrow J$ and isomorphisms $\phi_i:\C^\A_i\rightarrow \C^\B_{\pi(i)}$
of connected components considered as subalgebras such that
$\phi=\bigcup_{i\in I} \phi_i$. Conversely, any choice of $\pi$ and $\phi_i$ yields an isomorphism.
From this, we have:

\begin{lemma}
\label{le:coniso1}
Let $\A$ and $\B$ be two isomorphic unary algebras, with connected components $\{ C^\A_i\::\: i\in I\}$, $\{C^\B_i\::\: i\in I\}$ respectively.
Suppose that $I_0\subseteq I$ is a finite set, that $\pi : I_0\rightarrow I$ is an injection, and that
$\phi_i:\C^\A_i\rightarrow \C^\B_{\pi(i)}$ ($i\in I_0$) are isomorphisms of unary subalgebras.
Then there is an isomorphism $\phi:\A\rightarrow\B$ which simultaneously extends all $\phi_i$ ($i\in I_0$).
\end{lemma}

The reflexive, transitive closure $\rightarrow^\ast$ of the adjacency relation $\rightarrow$ in $\Gamma(\A)$ is a preorder on $A$,
noting that if $x\rightarrow^\ast y$ we consider $x$ as the larger of the two.
The associated equivalence relation $\leftrightarrow^\ast$ is given by
$a\leftrightarrow^\ast b$ if and only if $a\rightarrow^\ast b$ and $b\rightarrow^\ast a$. The equivalence classes of this relation are called the \emph{strongly connected components} of $\Gamma(\A)$ (and of $\A$). The preorder 
$\rightarrow^\ast$ on $A$ induces a partial order on the quotient $A/\!\!\leftrightarrow^\ast$. The maximal elements in this order are called the \emph{top components} of $\A$, and the minimal elements are called the \emph{bottom components} of $\A$.
Strongly connected components generally are not subalgebras; however, the bottom components are. Furthermore, if $A$ is finite and all $f\in F$ are bijections, the strongly connected components coincide with connected components, and hence are all subalgebras.

If $\A$ and $\B$ are isomorphic unary algebras then of course $\Gamma(\A)$ and $\Gamma(\B)$ are isomorphic digraphs. The converse need not hold, as a typical edge $a\rightarrow f(a)$ in $\Gamma(\A)$ `forgets' about $f$.
The converse does hold in the case of \emph{mono-unary} algebras, i.e. when there is only a single basic operation.

We now describe another isomorphism invariant of unary algebras, which we will make use of in Section \ref{sec:uncountable}.
Suppose $\A$ is a connected unary algebra with a bottom component $D$.
Consider the induced subgraph of $\Gamma(\A)$ on the set $A\setminus D$.
The connected components of this subgraph will be called the \emph{outer sections} of $\Gamma(\A)$ (and of $\A$) 
with respect to $D$. If $D$ is a unique bottom component of $\A$, the reference to it may be omitted.
 Now suppose that $\phi:\A\rightarrow\B$ is an isomorphism. Then $\B$ must also be connected and 
 $E:=\phi(D)$ is a bottom of component in $\B$. But then $\phi(A\setminus D)=B\setminus E$, and it follows that $\phi$ induces a (graph) isomorphism between the induced subgraph of $\Gamma(\A)$ on $A\setminus D$ and the induced subgraph of $\Gamma(\B)$ on $B\setminus E$. This induced isomorphism must preserve the connected components of the two graphs, and we obtain the following:

\begin{lemma}
\label{la:outer}
Let $\A$ and $\B$ be two connected unary algebras and let $D$ be a bottom component of $\A$.
If $\phi:\A\rightarrow\B$ is an isomorphism, then there is a $1-1$ correspondence between the outer sections of $\A$ with respect to $D$ and the outer sections of $\B$ with respect to $\phi(D)$ such that the pairs of corresponding outer sections are isomorphic digraphs.
\end{lemma}

The \emph{monoid} $\Mon(\A)$ of a unary algebra $\A=(A,F)$ is the monoid of transformations on the set $A$ generated by $F$.  Alternatively, $\Mon(\A)$ is the unary part of the clone of $\A$; see \cite[Section 4.1]{alv87}. The clone itself consists of compositions of the elements of $\Mon(\A)$ and projections.

Consider now the cartesian product $\A^\N$.
We will view the elements of $\A^\N$ as infinite tuples $\x=(x_1,x_2,x_3,\dots)$, $x_i\in A$.
For such a tuple define its \emph{content} as the set $\cont(\x):=\{ x_i\::\: i\in\N\}$, and its \emph{format} as the equivalence relation $\form(\x):=\{(i,j)\in\N\times\N\::\: x_i=x_j\}$.
If we view $\x$ as a mapping $\N\rightarrow A$ then $\cont(\x) $ and $\form(\x)$ are its image and kernel respectively. 
With this in mind the following is immediate:

\begin{lemma}
\label{le:formsub}
For a unary algebra $\A$, an element $\x\in A$ and $f\in \Mon(\A)$ we have $\form(\x)\subseteq \form(f(\x))$; the equality holds if and only if $f$ is injective on $\cont(\x)$.
\end{lemma}

A typical equivalence class of $\form(\x)$ will be written as $[a]_\x:=\{ i\in\N\::\: x_i=a\}$ where $a\in \cont(\x)$.

It is well-known that the variety generated by a finite algebra is locally finite \cite[Theorem 10.16]{bs81}, and hence has only finitely many $k$-generated members for any fixed finite $k$. Specialising this for our purposes, we record the following result, which can also be easily proved directly:

\begin{lemma}
\label{le:monfin}
For a finite unary algebra $\A$ there are only finitely many monogenic subalgebras of $\A^\N$ up to isomorphism.
\end{lemma}


\section{Unary algebras of countable type}
\label{sec:countable}

In this section we prove the `if' part of the Main Theorem. 
As a prelude, we make an observation, which will help streamline the proof,
and will also be helpful in the next section.
For a unary algebra $\A=(A,F)$, let $\A_c$ be the new algebra on the same set obtained by adding all constant unary operations to $F$;
specifically, $\A^c=(A,F\cup F^c)$, with $F_c=\{ f_a\::\: a\in A\}$, $f_a(x)=a$ for all $x\in A$.

\begin{lemma}
\label{la:extc}
A finite unary algebra $\A$ is of countable type if and only if $\A_c$ is of countable type.
\end{lemma}

\begin{proof}
($\Rightarrow$) This direction actually holds for arbitrary finite algebras. 
Suppose, aiming for contradiction, that $\A$ is of countable type, while $\A_c$ is of uncountable type.
There exists an uncountable collection $\{ \B_i\::\: i\in I\}$ of mutually non-isomorphic countable subdirect powers  in 
$\A_c^\N$ such that their reducts $\B_i'$, $i\in I$, with respect to the operations $F$, are all isomorphic.
Notice that each $B_i$ contains the diagonal $D:=\{(a,a,\dots)\::\: a\in A\}$.
Fix $i_0\in I$ and isomorphisms $\phi_i : \B_i'\rightarrow \B_{i_0}'$, $i\in I$.
Since $D$ is finite and all the $B_i$ countable, there exist $i,j\in I$ such that $\phi_i(d)=\phi_j(d)$ for all $d\in D$.
But then $\phi_j^{-1}\phi_i:B_i\rightarrow B_j$ is an isomorphism between $\B_i'$ and $\B_j'$, and it fixes all the elements of $D$. Therefore $\phi_j^{-1}\phi_i$ is an isomorphism between $\B_i$ and $\B_j$, a contradiction.

($\Leftarrow$)
Now suppose, again aiming for contradiction, that $\A$ is of uncountable type and $\A_c$ is of countable type.
Let $\{ \B_i\::\: i\in I\}$ be an uncountable collection of pairwise non-isomorphic countable subdirect powers in $\A^\N$.
For $i\in I$, let $\B_i'$ be the subalgebra of $\A_c^\N$ with the carrier set $B_i\cup D$.
(Notice that this step is not valid for non-unary algebras.)
Since $D$ is finite and $\A_c$ of countable type, there exist 
$i,j\in I$ such that $\B_i'\cong \B_j'$ and $B_i\cap D= B_j\cap D$.
Let $\phi: \B_i'\rightarrow\B_j'$ be an isomorphism.
Notice that $\phi(d)=d$ for all $d\in D$.
It follows that $\phi(B_i)=B_j$, and hence $\phi$ is an isomorphism between $\B_i$ and $\B_j$, a contradiction.
\end{proof}

\begin{prop}
\label{prop:mainif}
If in a finite unary algebra $\A=(A,F)$ every $f\in F$ is a bijection or a constant mapping then $\A$ is of countable type.
\end{prop}

\begin{proof}
By applying Lemma \ref{la:extc} twice, we may assume without loss of generality that $F$ consists of bijections only.
But then, every monogenic subalgebra of $\A^\N$ is a strongly connected component. 
By Lemma \ref{le:monfin} there are only finitely many non-isomorphic strongly connected components, say $\C_1,\dots,\C_k$.
Every subalgebra of $\A^\N$ is a union of monogenic subalgebras, and hence a disjoint union of 
strongly connected components.
Thus the isomorphism type of such a subalgebra is determined by a sequence $(m_1,\dots, m_k)$,
where  $m_i\in\N\cup\{0,\aleph_0\}$ is the number of occurences of $\C_i$ in the disjoint union.
There are only countably many such sequences and the result follows.
\end{proof}

\begin{rem}
In fact, a much stronger statement is true than that used in the above deduction regarding the case where $F$ consists of bijections only: the variety generated by $\A$ contains only countably many countable members up to isomorphism; this is observed, for instance, in the introduction of \cite{hart94}.
\end{rem}

%

An immediate interesting consequence of Proposition \ref{prop:mainif} is the following:

\begin{cor}
Every unary algebra of size $2$ is of countable type.
\end{cor}

\section{Unary algebras of uncountable type}
\label{sec:uncountable}

This section is devoted to proving the `only if' part of the Main Theorem:

\begin{prop}
\label{pro:oif}
A finite unary algebra $\A=(A,F)$ such that $F$ contains an operation which is neither a bijection nor a constant mapping is of uncountable type.
\end{prop}

\begin{proof}
Let us suppose that
\[
|A|=n,\ A=\{a_1,\dots,a_n\}.
\]
Let $\fmin\in\Mon(\A)$ be a non-constant operation with minimal possible image size. Thus $1<|\fmin(A)|<n$, and without loss of generality we may assume that
\begin{equation}
\label{eq:fan}
\fmin(a_{n-1})=\fmin(a_n).
\end{equation} 

\begin{claim}
\label{cl:inteqs}
There exist equivalence relations $\sigma_1,\sigma_2,\dots$ on the set $\N$ such that the following hold:
\begin{thmenumerate}
\item
\label{it:ie1}
each $\sigma_k$ has a finite number $c_k$ of equivalence classes;
\item
\label{it:ie2}
$n\leq c_1<c_2<\dots$;
\item
\label{it:ie3}
if $C$ is an equivalence class of some $\sigma_k$, and  if $D$ is an equivalence class of some $\sigma_l$ with $l\neq k$, then
$C\cap D\neq\emptyset$.
\end{thmenumerate}
\end{claim}

\begin{proof}
We can take $\sigma_k$ to be the congruence modulo $p_k$, where $n\leq p_1<p_2<\dots$ is an increasing sequence of primes.
\end{proof}

We fix one choice of $\sigma_1,\sigma_2,\dots$ as above for the rest of the proof of the proposition.

Consider an arbitrary $k\in \mathbb{N}$, and 
 denote the equivalence classes of $\sigma_k$ by $C_{k,i}$, $i=1,\dots,c_k$.
 For each $l=0,\dots,c_k-n$ define a 
tuple $\t_{k,l}\in\A^\N$ with full content $A$, and
\begin{equation}
\label{eq:tkl}
[a_i]_{\t_{k,l}}=
\begin{cases}
C_{k,i} & \text{if } 1\leq i\leq n-2\\
C_{k,n-1}\cup\dots\cup C_{k,n-1+l} & \text{if } i=n-1\\
C_{k,n+l}\cup\dots\cup C_{k,c_k} & \text{if } i=n.
\end{cases}
\end{equation}
Notice that 
\begin{equation}
\label{eq:sftkl}
\sigma_k\subseteq \form(\t_{k,l})\quad\text{for all } l=0,\dots, c_k-n.
\end{equation}
Also observe that
the only places in which two tuples $\t_{k,l}$ and $\t_{k,m}$ differ are those where one has $a_{n-1}$ and the other $a_n$. Since $\fmin(a_{n-1})=\fmin(a_n)$ it follows that 
\begin{equation}
\label{eq:fmintkl}
\fmin (\t_{k,l})=\fmin (\t_{k,m})\quad \text{for all } l,m\in\{0,\dots,c_k-n\}.
\end{equation}
In particular, the subalgebra
\[
\T_k:=\langle \{ \t_{k,l}\::\: l=0,\dots,c_k-n\}\rangle\leq \A^\N
\]
is connected.

\begin{claim}
\label{cl:Tk}
$\T_k$  has exactly $c_k-n+1$ top  components.
\end{claim}

\begin{proof}
In any unary algebra with a given set of generators each top component must contain a generator.
For $l\neq m$ the formats of $\t_{k,l}$ and $\t_{k,m}$ are incomparable via inclusion.
Therefore, by Lemma \ref{le:formsub}, we have $f(\t_{k,l})\neq \t_{k,m}$ for all $f\in\Mon(\T_k)$.
Therefore the strongly connected component of each generator $\t_{k,l}$ is a top component,
and these are all distinct.
\end{proof}

It follows immediately from the previous claim that the algebras $\T_k$ are pairwise non-isomorphic. We will now 
use these algebras as the building blocks from which to construct uncountably many non-isomorphic subdirect powers of $\A$.

To motivate what follows, let us briefly consider the case where the $\T_k$ are pairwise disjoint.
One can see from what follows that this is the case precisely when $\Mon(\A)$ contains no constant mappings.
For any $K\subseteq \N$ let $S_K:=\bigcup_{k\in K} T_k$. Then this is the carrier of a subalgebra $\S_K$ of $\A^\N$.
The connected components of $\S_K$ are precisely $\T_k$, $k\in K$.
Therefore, for $K\neq L$ we have $\S_K\not\cong \S_L$, because the $\T_k$ are pairwise non-isomorphic.
Two issues arising from this argument is that the $\T_k$ need not be disjoint, and that $\S_K$ may not be subdirect.

We now deal with these issues, and hence with the general case. 
By Lemma \ref{la:extc}, we may assume without loss of generality that all constant operations are among the basic operations of $\A$.
Let $D:=\{(a,a,a,\dots)\::\: a\in A\}$, the diagonal of $\A^\N$.

\begin{claim}
\label{cl:Tinter}
For any distinct $k,l\in\N$ we have $T_k\cap T_l=D$.
\end{claim}

\begin{proof}
We have $D\subseteq T_k$ for all $k$ because of the presence of the constant operations.
Now consider an arbitrary $\x\in T_k\cap T_l$. Write $\x=f(\t_{k,i})=g(\t_{l,j})$, where $0\leq i\leq c_k-n$, $0\leq j\leq c_l-n$, $f,g\in\Mon(\A)$.
By Lemma \ref{le:formsub} and \eqref{eq:sftkl} we have
\[
\sigma_k\subseteq \form(\t_{k,i})\subseteq \form(\x)\quad\text{and}\quad
\sigma_l\subseteq \form(\t_{l,j})\subseteq \form(\x).
\]
Hence $\form(\x)$ contains $\sigma_k\vee\sigma_l$, the smallest equivalence relation on $\N$ containing both $\sigma_k$ and $\sigma_l$. But $\sigma_k\vee\sigma_l$ is in fact the full relation on $\N$, because of Claim \ref{cl:inteqs} \ref{it:ie3}.
Hence $\x\in D$, 
yielding $T_k\cap T_l\subseteq D$, and completing the proof of the claim.
\end{proof}

\begin{claim}
\label{cl:outTk}
The subgraph of $\overline{\Gamma}(\T_k)$ induced on $T_k\setminus D$ is connected.
\end{claim}

\begin{proof}
By \eqref{eq:fmintkl} we have
$\fmin (\t_{k,i})=\fmin (\t_{k,j})$ for all $i,j\in\{ 0,\dots, c_k-n\}$.
Since $\cont(\t_{k,i})=A$ and $\fmin$ is not a constant mapping, we have $\fmin (\t_{k,i})\not\in D$, 
and the claim follows.
\end{proof}

Now, for any subset $K\subseteq \N$ let us define
$S_K:=\bigcup_{k\in K} T_k$.
This is the carrier set for a subalgebra of $\A^\N$, and we denote this subalgebra by $\S_K$.
Since $D\subseteq S_K$ it follows that $\S_K$ is in fact a subdirect power of $\A$.

\begin{claim}
\label{cl:SKSL}
For any two distinct subsets $K,L\subseteq \{2n,2n+1,2n+2,\dots\}$ we have $\S_K\not\cong \S_L$.
\end{claim}

\begin{proof}
By Claim \ref{cl:Tinter}, the outer sections of $\S_K$ (resp, $\S_L$) are $T_k\setminus D$ for $k\in K$
(resp. $T_l\setminus D$, $l\in L$).
 Without loss we may assume that there exists $k\in K\setminus L$.
 By Claim \ref{cl:Tk}, the outer section $T_k\setminus D$ of $\S_K$ has $c_k-n+1$ top components.
 By the same claim, and Claim \ref{cl:inteqs} \ref{it:ie1}, \ref{it:ie2} none of the outer sections $T_l\setminus D$ ($l\in L$) of $\S_L$ have exactly that number of top components. 
 Therefore $\S_K\not\cong \S_L$, as required.
\end{proof}

This exhibits uncountably many subdirect powers of $\A^\N$, and completes the proof of the proposition.
\end{proof}

\section{Remarks on infinite cardinals}\label{sec:infinite}

It is of note that the result of McKenzie in~\cite{mckenzie1982subdirect} deals with subdirect powers of arbitrary non-abelian groups, and subdirect powers of cardinality $\kappa$ greater than or equal to the order of the group.
In this section we discuss examples chosen to indicate that our 
Main Theorem may not extend to a similarly concise result. 
We begin by characterising countable type
for countable algebras with a single bijective basic operation. 
We then discuss arbitrary cardinalities of subdirect powers, and show how the unary algebras differ from groups in this regard.

\begin{prop}\label{pro:infbij}
Let $(A,f)$ be a countable monounary algebra where $f$ is a bijection. Then $\A$ is of countable type if and only if there is a bound on the length of finite cycles contained in $\A$.
\end{prop}

\begin{proof}
We begin by discussing the structure of $\A^\N$, so that we can classify subdirect powers in terms of their connected components. Let $\a=(a_1,a_2,\dots)\in\A^\N$, and define $\Seq(\a):=(s_1,s_2\dots)$ where $s_i$ is the length of the cycle containing $a_i$. If $\Seq(\a)$ is bounded, then for $m=\lcm(\Seq(\a))$
 we have that
 $\a\rightarrow f(\a)\rightarrow \dots\rightarrow f^{m-1}(\a)\rightarrow f^m(\a)=\a$ is a cycle of length $m$.
 If $\Seq(\a)$ is not bounded, i.e. if $\cont(\a)$ either contains elements from infinitely many finite cycles with unbounded length, or it contains at least one element from an infinite cycle, then $f^l(\a)\neq\a$ for all $l\in\N$. Additionally, since $f$ is a bijection, $f^{-l}(\a)\neq\a$ for all $l\in\N$, and so the connected component of $\a$ is isomorphic to $(\Z,x\mapsto x+1)$. Since the connected components of a subdirect power are subalgebras of connected components of the direct power, the only options are finite cycles, $(\Z,x\mapsto x+1)$ and $(\N,x\mapsto x+1)$, and there are either none or uncountably many of each (apart from possibly cycles of length one).

($\Leftarrow$)
Assume $\A$ has a bound on the length of its finite cycles.
It follows from the foregoing analysis that the lengths of finite cycles in $\A^\N$ are also bounded, and thus there are only finitely many isomorphism types of connected subalgebras of $\A^\N$. 
Using the fact that every unary algebra is a disjoint union of its connected components, we conclude
that there are countably many countable subdirect powers.

($\Rightarrow$)
Assume now that $\A$ contains infinitely many cycles of different lengths, and
 we will show that $\A$ is of uncountable type. We first exhibit a subdirect power which contains no cycles. Let $(y_1,y_2,\dots)$ be a sequence of elements of $\A$ such that each element is contained in a finite cycle, and the sequence of respective cycle lengths is strictly increasing.
We define for $a\in A$ the tuples $\alpha_a$ and $\beta_a$ as follows:
\[\alpha_a:=(a,y_1,a,y_2,a,y_3\dots),\]
\[\beta_a:=(y_1,a,y_2,a,y_3,a\dots).\]
Note $\alpha_a$ has $a$ in the odd coordinates, and the sequence $(y_1,y_2,\dots)$ in the even coordinates, and $\beta_a$ is defined analogously. Then define the set $S\subseteq\A^\N$ as $S:=\bigcup_{a\in A}\{\alpha_a,\beta_a\}$.
Any subalgebra which contains this set is subdirect. In particular the subalgebra $\langle S\rangle$ is subdirect. 
The elements of $S$ are contained in distinct infinite cycles. 
Hence $\langle S\rangle$ has countably many connected components all isomorphic to $(\N,x\mapsto x+1)$,
specifically, $\langle \gamma\rangle$  for each $\gamma\in S$.
Note that for $\y_i:=(y_i,y_i,\dots)$, the connected component of $\y_i$ is a cycle of the same length as the cycle
in $\A$ containing $y_i$. Then for any subset $Y\subseteq\{\y_1,\y_2,\dots\}$ the algebra $\T_Y:=\langle S\cup Y\rangle$ is a countable subdirect power. For $Y,Y'$ distinct subsets, $\T_Y\not\cong\T_{Y'}$, as the sets of finite cycles they contain are distinct. Thus if $\A$ contains infinitely many cycles, it is of uncountable type.
\end{proof}

Proposition \ref{pro:infbij} implies that the restriction to finite algebras in the statement of Proposition~\ref{pro:oif} is necessary. It is not the case that any countable unary algebra whose operations are all bijections or constant functions is of countable type. One might wonder if Proposition \ref{pro:infbij} extends to unary algebras in general. Two natural extensions which one might consider are that an algebra $(A,F)$ where $F$ contains only bijections is of countable type if and only if: a) the size of finite connected components is bounded; or b) for all $f\in F$ the length of finite cycles in $f$ is bounded. However, neither of these is true, as demonstrated by the following example.

\begin{example} \label{exa:transpositions}
Let $\T:=(\N,F)$, where $F$ is the set of all transpositions on $\N$ i.e. bijections which fix all but two points. We show that the algebra $\T$ is of uncountable type.

Let $\sigma$ be a format with finitely many classes, say $n$. Note that for two tuples $\x,\y\in\T^\N$ with format $\sigma$ the length of the shortest path between them in $\overline{\Gamma}(\T^\N)$ is at most $n$ (this can be proved easily by induction on $n$) and there exist such tuples for which the shortest path 
is of length precisely $n$, e.g. when their contents are disjoint.
It is then clear that the set $C_\sigma\subseteq\T^\N$ of tuples with format $\sigma$ is strongly connected and subdirect. Thus the set $\mathscr{C}=\{C_{\sigma_i}:i\in\N\}$ where $\sigma_i$ is a format with $i$ format classes is a countable collection of non-isomorphic connected countable subdirect powers of $\T$. Taking 
unions over arbitrary subsets of $\mathscr{C}$ then gives us uncountably many non-isomorphic subdirect products.
\end{example}

We now turn our attention to larger cardinalities of subdirect powers. 
Whereas, in the group case, McKenzie \cite{mckenzie1982subdirect} obtained exactly $2^\kappa$ subdirect products of cardinality $\kappa$, we exhibit an example of a finite unary algebra which is of uncountable type, but can have strictly less than $2^\kappa$ subdirect products.

\begin{prop}\label{pro:2linekappa}
Let $\A:=(A,f) $ be the monounary algebra with $A=\{0,1,2\}$ and $f(x)=\max(x-1,0)$, and let $\kappa$ be any infinite cardinal. The number of non-isomorphic subdirect powers of $A$ of cardinality $\kappa$ is $2^\alpha$, where $\alpha$ is the number of cardinals $\beta\leq\kappa$.
\end{prop}

\begin{proof}
For this proof we extend the definition of content and format to arbitrary cardinal powers.
We first describe the full product $\A^\gamma$ for any infinite cardinal $\gamma$. Since $f^2$ is the constant map on $\A$ with value $0$, it follows that $\Gamma(\A^\gamma)$ is a directed tree of depth $2$, with the sink 
$\0$, the diagonal element corresponding to 0.
The predecessors of $\0$ are the tuples $\x\in A^\gamma$ such that $\cont(\x)=\{0,1\}$.
Consider now an arbitrary such $\x$. Its predecessors are all tuples $\y\in A^\gamma$ such that
$\cont(\y)=A$, $[0]_\y\cup [1]_\y=[0]_\x$ and $[2]_\y=[1]_\x$. In particular, $\x$ has $2^{|[0]_\x|}$ predecessors.

Thus any subalgebra of cardinality at most $\kappa$ must consist of $\0$, at most $\kappa$ elements with content $\{0,1\}$, and for each such point at most $\kappa$ of its predecessors. Thus we can determine the isomorphism type of a subalgebra $\S\leq\A^\gamma$ by the number of predecessors of $\0$ with $\beta$ predecessors for every cardinal $\beta\leq\kappa$. Letting $\gamma$ be sufficiently large (in particular, such that there are at least $\alpha^\alpha$ elements with $\kappa$ predecessors), there are then  $\alpha^\alpha$ subalgebras of cardinality at most $\kappa$. This gives us an upper bound for the number of subdirect products of cardinality $\kappa$.
If we instead insist on selecting subalgebras which contain the diagonal, and precisely $\kappa$ elements outside the diagonal with exactly one predecessor in the subalgebra, then we can follow the same argument to obtain $\alpha^\alpha$ non-isomorphic subdirect powers of cardinality $\kappa$, which is a lower bound for the number of subdirect powers of cardinality $\kappa$. Thus there are exactly $\alpha^\alpha=2^\alpha$ subdirect powers of $\A$ of cardinality $\kappa$.
\end{proof}

This shows that the result of McKenzie, that non-abelian groups have $2^\kappa$ subdirect powers of cardinality $\kappa$ (for $\kappa\geq|G|$), will not generalize to arbitrary algebraic structures. In particular, if we assume the continuum hypothesis, then the number of subdirect powers of $\A$ of cardinality $\aleph_1$ is $2^{\aleph_0}<2^{\aleph_1}$.

\section{Further remarks: abelian property and boolean separation}
\label{sec:conclusion}

As we mentioned in the introduction, the result of  Hickin and Plotkin \cite{hickin1981boolean} implies that a finite group is of countable type if and only if it is abelian. There is a natural notion of being abelian for general algebras; see \cite[Section 4.13]{alv87}.
The question then arises as to whether being abelian and being of countable type are equivalent for some other or broader varieties than that of groups.
In particular, the question as to whether the two notions are equivalent in any congruence permutable variety seems well worth investigating.
Intriguingly, the investigation into commutative semigroups \cite{CRinpr} leaves open the possibility of two notions being equivalent in the variety of semigroups.
On the other hand, our main result here shows that the two notions are certainly not equivalent for unary algebras.
Indeed, every unary algebra is abelian, as can be easily seen by inspecting \cite[Definition 4.146]{alv87}, but we have seen that there exist unary algebras of uncountable type.

Another interesting parallel is that between  the number of subdirect powers and separation in boolean powers.
Let $\B$ be a boolean algebra, represented as an algebra of subsets of a set $X$, and let $\A$ be a finite algebra.
The \emph{boolean power} $\A^\B$ is the subdirect power of $\A^X$ consisting of all $f:X\rightarrow A$ for which each equivalence class of the format is an element of $\B$.
Note that if $\B$ is countable, then so is $\A^\B$.
For a more systematic introduction into boolean powers see \cite[Section IV.5]{bs81}.
We say that the algebra $\A$ is \emph{boolean separating} if for any non-isomorphic boolean algebras $\B_1$, $\B_2$ we have that $\A^{\B_1}\not\cong\A^{\B_2}$.
Clearly, if $\A$ is boolean separating then it is of uncountable type, since there are uncountably many countable boolean algebras.
Lawrence \cite{lawrence1981boolean} proved that a finite subdirectly irreducible group is boolean separating if and only if it is non-abelian.
Further similar results are obtained in \cite{hickin1981boolean,mckenzie1982subdirect}, and a complete description of finite 
boolean separating groups is obtained by Apps in \cite{ap82}.
Combined with our discussion in the first paragraph, Lawrence's result implies that a finite subdirectly irreducible group is boolean separating if and only if it is of uncountable type. Yet again a question arises whether this may hold in greater generality. And yet again, unary algebras provide a counterexample.

\begin{prop}
\label{pro:bsep}
Let $\A:=(A,f) $ where $A=\{0,1,2\}$ and $f(x)=\max(x-1,0)$.
If $\B_1$ and $\B_2$ are any two countable boolean algebras of subsets of $\N$ which contain all finite sets, then
 $\A^{\B_1}\cong \A^{\B_2}$.
In particular, $\A$ is subdirectly irreducible, is not boolean separating, but is of uncountable type.
\end{prop}

\begin{proof}
Since the algebras $\A^{\B_1}$ and $\A^{\B_2}$ are mono-unary, we have that $\A^{\B_1}$ and $\A^{\B_2}$ are isomorphic if and only if 
the directed graphs $\Gamma(\A^{\B_1})$ and $\Gamma(\A^{\B_2})$ are isomorphic.
Let $\B$ stand for either of $\B_1$ or $\B_2$. We are going to describe the graph $\Gamma(\A^\B)$, and it will transpire from this description that
the graph does not depend on the algebraic structure of $\B$, from which it will follow that $\A^{\B_1}\cong\A^{\B_2}$.

Our description of $\Gamma(\A^\B)$ is analogous to what was seen in the proof of Proposition~\ref{pro:2linekappa}. We have a sink 
$\0:=(0,0,0,\dots)$,
its predecessors are the countably many tuples $\x$ from $\A^\B$ with $\cont(\x)=\{0,1\}$.
The predecessors of a given such $\x$ are all tuples $\y\in A^\N$ such that
$\cont(\y)=A$, $[0]_\y\cup [1]_\y=[0]_\x$ and $[2]_\y=[1]_\x$.
In particular, since $\B$ contains all finite sets, $\x$ has finitely many predecessors if and only if $[0]_\x$ is finite, in which case $\x$ has precisely $2^k$ predecessors, where $\bigl|[0]_\x\bigr|=k$.
For any $k\neq0$ there are countably many such $\x$, again because $\B$ contains all finite sets. There is precisely one $\x$ with $k=0$, namely $\1=(1,1,1,\dots)$.
The remaining countably many $\x$, with $[0]_\x$ infinite, all have countably many predecessors.
To sum up, $\Gamma(\A^\B)$ has a unique sink, which has countably many predecessors; of these predecessors, there is precisely one with a single predecessor, and there are countably many with $l$ predecessors, for each $l\in \{2^k\::\: k=2,3,\dots\}\cup\{\aleph_0\}$.
This completely describes the graph $\Gamma(\A^\B)$, and completes the proof of the isomorphism.

The above analysis shows that $\A$ is not boolean separating.
That $\A$ is subdirectly irreducible follows from the fact that it has a single proper congruence, namely the equivalence relation with classes $\{0,1\}$ and $\{2\},$ using the standard criterion for subdirect irreducibility \cite[Theorem 4.40]{alv87}. Finally, that $\A$ is of uncountable type follows immediately from our Main Theorem.
\end{proof}

\textbf{Acknowledgement.}
The authors are grateful to an anonymous referee for their helpful suggestions. Specifically, one of these suggestions has led us to explicitly formulate Lemma \ref{la:extc}, and streamline proofs of both directions of the Main Theorem.

\bibliographystyle{plain}

\end{document}